\documentclass[10pt]{amsart}
\usepackage{amsmath,epsf,hyperref}
\usepackage{amssymb,stmaryrd}
\usepackage{tikz} \usetikzlibrary{matrix,arrows,decorations.pathmorphing}

\title{Multiplicity in difference geometry}
\date{\today}

\author{Ivan Toma{\v s}i{\'c}}
\address{Ivan Toma{\v s}i{\'c}\\
         School of Mathematical Sciences\\
  	Queen Mary University of London\\
         London, E1 4NS\\
        United Kingdom}
\email{i.tomasic@qmul.ac.uk}
\theoremstyle{plain}
\newtheorem{theorem}{Theorem}[section]
\newtheorem{corollary}[theorem]{Corollary}
\newtheorem{proposition}[theorem]{Proposition}
\newtheorem{lemma}[theorem]{Lemma}

\theoremstyle{definition}
\newtheorem{definition}[theorem]{Definition}

\theoremstyle{remark}
\newtheorem{remark}[theorem]{Remark}

\providecommand{\OO}{\mathcal{O}}

\providecommand{\Z}{\mathbb{Z}}

\providecommand{\M}{\mathfrak{m}}
\providecommand{\N}{\mathbb{N}}
\providecommand{\p}{\mathfrak{p}}
\providecommand{\q}{\mathfrak{q}}

\providecommand{\LL}{\mathbb{L}}

\providecommand{\dl}{\mathop{\rm dl}\nolimits}
\providecommand{\dt}{\mathop{\rm dimtot}\nolimits}
\providecommand{\dte}{\mathop{\rm dimtoteff}\nolimits}
\providecommand{\sd}{\mathop{\sigma\text{\rm-dim}}\nolimits}
\providecommand{\trdeg}{\mathop{\rm tr.deg}\nolimits}
\providecommand{\dd}{{\mathbf{d}}}
\providecommand{\dde}{{\mathbf{d}_{\rm eff}}}
\providecommand{\kk}{\mathbf{k}}
\providecommand{\specd}{{\rm Spec}^\sigma}
\providecommand{\spec}{{\rm Spec}}

\providecommand{\Cl}{\text{\rm Cl}}

\providecommand{\Div}{\text{\rm Div}}

\providecommand{\Hom}{\text{\rm Hom}}

\providecommand{\ztild}[1]{\rlap{$\smash{\tilde{\phantom{#1}}}$}\rlap{$\mathring{\phantom{#1}\kern1.1ex}$}#1\kern.1ex}

\providecommand{\lexp}[2]{{\vphantom{#2}}^{#1}{\kern-.1ex#2}}

\providecommand{\acirc}[1]{\phantom{a}\llap{$\scriptstyle#1$}
\kern.01ex\lower.75ex\hbox{$\smash{\mathring{}}$}}

\providecommand{\lzexp}[3]{{\vphantom{#2}}^{\lower0.0ex\hbox{\smash{$\acirc{#1}$}}}\kern-.1ex #2}

\providecommand{\lrexp}[3]{{\vphantom{#2}}^{#1}{\kern-.1ex#2^#3}}

\begin{document}

\begin{abstract} 
We prove a first principle of preservation of multiplicity in difference geometry, paving the way for the development of a more general intersection theory. In particular,
the fibres of a $\sigma$-finite morphism between difference curves are all of
the same size, when counted with correct multiplicities.
\end{abstract}

\maketitle

\section{Introduction}

Unlike in algebraic geometry, where the goal of intersection theory is quite
well-defined and understood, in difference geometry, due to a much
richer class of varieties, and a wider range of possible dimensions, there
are several levels at which we can pose the question of  the existence
of an appropriate theory of multiplicity or intersection theory.

In difference geometry, we have two notions of dimension, the \emph{transformal dimension} and 
the \emph{total dimension}. Total dimension only makes sense (is finite) when
transformal dimension is $0$, and is more closely related to the usual notions
of dimension such as Krull dimension or transcendence degree. 

The first possibility for the intersection theory is the following problem.
If we have two difference subschemes of complementary transformal dimensions in a given ambient space, their intersection will be of finite total dimension, and,
as \cite{ive-mark} shows, it makes sense to ask about its \emph{size}. 
The first hint that a systematic study of this kind of intersection theory
may be possible was given in \cite{ive-diffmeas}, and the author is
informed that a substantial piece of work in this direction is Gabriel Giabicani's thesis.

Another possibility, dealing with object closer in size (and nature) to 
algebraic varieties, but much more mysterious, is to try and intersect
difference schemes of complementary total dimension. Unfortunately,
in the na\"ive setting of difference schemes in a strict sense, the points
of intersection are blatantly missing, and there is no hope of a smooth theory.
The new idea of this paper is that in the context of  \emph{generalised difference
schemes} the multiplicity principles actually work. 

The author hopes that
these results will serve as a foundation for a whole new Intersection Theory 
in Difference Algebraic Geometry.

One of the most remarkable revelations for the author was that 
the theory of divisors on non-singular difference curves ties in neatly with the existing
theory of the divisor class groups of Krull (and Pr\"ufer) domains, the non-noetherian analogues of Dedekind domains.
\section{Generalised difference schemes}

For a more detailed account of the material of the present section, including proofs,
we refer the reader to \cite{ive-tgals}.

\begin{definition}
A \emph{generalised difference ring} is a pair
 $(A,\Sigma)$,
where $A$ is a commutative ring with identity, and $\Sigma$ is a 
set of monomorphisms $A\to A$
such that 
\begin{enumerate}
\item For every $\sigma, \tau\in\Sigma$,
there exists a (necessarily unique) $\sigma^{\tau}\in\Sigma$ such that
$$
\tau\circ\sigma^{\tau}=\sigma\circ\tau.
$$
\item It follows that $\sigma^\sigma=\sigma$ for every $\sigma\in\Sigma$.

\item We also require that for all $\sigma,\tau\in\Sigma$,
$$
(\sigma^{\tau})^\varphi=(\sigma^\varphi)^{({\tau}^\varphi)}.
$$
\end{enumerate}

A morphism $\varphi:(B,T)\to (A,\Sigma)$ consists of a ring morphism 
$\varphi:B\to A$ and a map $()^\varphi:\Sigma\to T$ such that 
\begin{enumerate}
\item
$$\varphi\circ\sigma^\varphi=\sigma\circ\varphi.$$
\item Moreover, we require that
$$
(\tau^{\sigma})^\varphi=(\tau^\varphi)^{({\sigma}^\varphi)}.
$$
\end{enumerate}
\end{definition}

\begin{definition}\label{defspecsigma}
Let $(R,\Sigma)$ be an object of a difference category over the category of
commutative rings with identity.
We shall consider each of the following subsets of $\spec(R)$ as locally ringed spaces
with the Zariski topology and the structure sheaves induced from $\spec(R)$:
\begin{enumerate}
\item $\spec^\sigma(R)=\{\p\in\spec(R):\sigma^{-1}(\p)=\p\}$, for any $\sigma\in\Sigma$;
\item $\spec^\Sigma(R)=\cup_{\sigma\in\Sigma}\spec^\sigma(R)$;
\item $\spec^{(\Sigma)}(R)=\cap_{\sigma\in\Sigma}\spec^\sigma(R)$.
\end{enumerate}
%
 In discussions of induced topology, we shall use the notation
 $V^\sigma(I)$, $D^\sigma(I)$, $V^\Sigma(I)$, $D^\Sigma(I)$, $V^{(\Sigma)}(I)$, $D^{(\Sigma)}(I)$
for the traces of $V(I)$ and $D(I)$ on $\spec^\sigma(R)$, $\spec^\Sigma(R)$, $\spec^{(\Sigma)}(R)$, respectively.
 \end{definition}
\begin{remark}\label{specsigma}
Let $(R,\Sigma)$ be a difference ring. Each $\sigma\in\Sigma$ induces
an endomorphism $\lexp{a}{\sigma}$ of the locally ringed space 
$(\spec^\Sigma(R),\OO_{\spec^\Sigma(R)})$. Thus we obtain a (generalised) difference object in the category of locally ringed spaces
$(\spec^\Sigma(R),\OO_{\spec^\Sigma(R)},\lexp{a}{\Sigma})$,
where $\lexp{a}{\Sigma}=\{\lexp{a}{\sigma}:\sigma\in\Sigma\}$.

 Thus we have a `contravariant' functor $\spec$ mapping an object $(R,\Sigma)$ to the object  $(\spec^\Sigma(R),\OO_{\spec^\Sigma(R)},\lexp{a}{\Sigma})$, and a morphism $(\varphi,()^\varphi):(S,T)\to(R,\Sigma)$ to the morphism
 $$(\lexp{a}{\varphi},\tilde{\varphi},()^\varphi):(\spec^\Sigma(R),\OO_{\spec^\Sigma(R)},\lexp{a}{\Sigma})\to(\spec^T(S),\OO_{\spec^T(S)},\lexp{a}{T}).$$ 
\end{remark}

\begin{definition}
Let $I$ be an ideal in a difference ring $(R,\Sigma)$. We say that:
\begin{enumerate}
\item $I$ is a $\Sigma$-ideal if $\sigma(I)\subseteq I$ for every $\sigma\in\Sigma$;
\item $I$ is \emph{$\Sigma$-well-mixed} if  $ab\in I$ implies $ab^\sigma\in I$ for any $\sigma\in\Sigma$;
\item $R$ itself is well-mixed if the zero ideal is;
\item $I$ is \emph{$\Sigma$-perfect} if for every $\sigma\in\Sigma$, $aa^\sigma\in I$ implies $a$ and $a^\sigma$ are both in $I$.
\end{enumerate}
\end{definition}
For a set $T$, we denote by $\{T\}_\Sigma$  the least $\Sigma$-perfect ideal
containing $T$.

\begin{proposition}
Let $(R,\Sigma)$ be a difference ring.
\begin{enumerate}
\item
$V^{(\Sigma)}(I)\subseteq V^{(\Sigma)}(J)$ if and only if $\{I\}_\Sigma\supseteq\{J\}_\Sigma$.
\item Let $f\in R$. Then $D^{(\Sigma)}(f)$ is quasi-compact. If $\Sigma$ is finite,
then $D^\Sigma(f)$ is quasi-compact.
\end{enumerate}
\end{proposition}

\begin{proposition}\label{wmaff}
Let $(A,\Sigma)$ be a well-mixed difference ring (or even ring with a set of monomorphisms),
$f\in A$. 
\begin{enumerate}
\item\label{jennn}
Both canonical morphisms
$$
A_{f_\Sigma}\to A_{\{f\}_\Sigma}\stackrel{\theta}{\to}\OO_{\spec^{(\Sigma)}A}(D^{(\Sigma)}(f)),
$$
are injective.
\newcounter{enumTemp}
    \setcounter{enumTemp}{\theenumi}
\end{enumerate}
If moreover $D^\Sigma(f)$ is quasi-compact, we have the following.
\begin{enumerate}\setcounter{enumi}{\theenumTemp}
\item \label{dvaaa}
For each $\bar{s}\in\OO_{X^\Sigma}(D^\Sigma(f))$, there exist $g_1,\ldots,g_r\in A$ such that
$D^\Sigma(f)=\cup_i D^\Sigma(g_i)$ and there is a section $s\in\OO_X(\cup_i D(g_i))$
such that $\bar{s}(x)=s(x)$ for $x\in D^\Sigma(f)$.
\item\label{dvaipol}
Let $\bar{s}\in\OO_{X^\Sigma}(D^\Sigma(f))$. The ideal
 $\mathop{\rm Ann}(\bar{s})=\{g\in A:g\bar{s}=0\}$ is well-mixed.
\item\label{triii}
Suppose $\bar{s}\in\OO_{X^\Sigma}(D^\Sigma(f))$ and $\p\in D^\Sigma(f)$ such that
$\bar{s}(\p)=0$. Then there is a $g\notin\p$ such that 
$g\bar{s}=0$ (on $D^\Sigma(f)$).
\item\label{cetriii}
Let $\bar{s}\in\OO_{X^\Sigma}(D^\Sigma(f))$ and $\p\in D^\Sigma(f)$. There exist
$g\notin\p$ and $a\in A$ such that $g\bar{s}-a=0$.
\item\label{pettt}
Let $\bar{s}\in\OO_{X^\Sigma}(X^\Sigma)$ such that $\bar{s}\restriction D^{(\Sigma)}(f)=0$.
Then there exists a $\nu\in\N[\sigma]$ such that $f^\nu\bar{s}=0$ (on $X^\Sigma$).
\item\label{sesttt}
There exist canonical injections $A\stackrel{i}{\hookrightarrow}\bar{A}=
\OO_{X^\Sigma}(X^\Sigma)\hookrightarrow\OO_{X^{(\Sigma)}}(X^{(\Sigma)})$ inducing an isomorphism
of difference schemes $(\lexp{a}{i},\tilde{\imath}):\spec^\Sigma(\bar{A})\to\spec^\Sigma(A)$.
\end{enumerate}
\end{proposition}

\begin{definition}
\begin{enumerate}
\item  An \emph{affine difference scheme} 
is an object $(X,\OO_X,\Sigma)$ of the difference category over the category of locally ringed spaces, which is
isomorphic to
 some $\spec^\Sigma(R)$ for some well-mixed $(R,\Sigma)$.
 \item
A \emph{difference scheme} 
is an object $(X,\OO_X,\Sigma)$ of the difference category over the category of locally ringed spaces, which is locally an affine difference scheme.
\item 
A \emph{morphism of difference schemes} 
$(X,\OO_X,\Sigma)\to(Y,\OO_Y,T)$ is just a morphism in the difference category over
the category of locally ringed spaces.
\end{enumerate}
\end{definition}

\begin{remark}
Given a difference scheme $(X,\OO_X,\Sigma)$ and $\sigma\in\Sigma$, we 
define a locally ringed space
$X^\sigma=\{x\in X:\sigma(x)=x\}$, together with the topology and structure sheaf 
induced from $(X,\OO_X)$. Since $\OO_{X^\sigma}:=\OO_X\restriction X^\sigma$, clearly 
$\sigma^\sharp\OO_{X^\sigma}\subseteq\OO_{X^\sigma}$ and  $(X^\sigma,\OO_{X^\sigma},\sigma)$ is a strict difference
scheme. We have the following properties:
\begin{enumerate}
\item $X=\bigcup_{\sigma\in\Sigma} X^{\sigma}$.
\item For every $\sigma, \tau\in \Sigma$,
there is a unique element $\sigma^{\tau}\in \Sigma$ such that
$$
{\tau}:(X^\sigma,\OO_{X^\sigma},\sigma)\to (X^{\sigma^{\tau}}, \OO_{X^{\sigma^{\tau}}},\sigma^{\tau})
$$
is a morphism of difference schemes in the strict sense.
\item
If $\varphi:(X,\OO_X,\Sigma)\to (Y,\OO_Y,T)$ is a morphism, then for every $\sigma\in \Sigma$ there exists a $\tau:=\sigma^\varphi\in T$
such that
$$
\varphi\restriction X^{\sigma}:(X^\sigma,\OO_{X^\sigma},\sigma)\to (Y^\tau,\OO_{Y^\tau},\tau)
$$
is a morphism of difference schemes in the strict sense. 
\end{enumerate}
\end{remark}

\begin{proposition}\label{embedcat}
The `global sections' functor $H^0$ is left adjoint to the contravariant functor 
$\spec$ from the category of well-mixed difference rings
to the category of difference schemes. For any difference scheme $(X,\Sigma)$ and
any well-mixed difference ring $(S,T)$ (with $T$ finite),
$$
\Hom\left((X,\Sigma),(\spec^T(S),T)\right)\stackrel{\sim}{\to}\Hom\left((S,T),(H^0(X),\Sigma)\right).
$$
\end{proposition}
\begin{remark}
It is worth remarking that, unlike in the algebraic case, $\spec^\sigma$ and $H^0$ do not
determine an equivalence of categories between the category of well-mixed difference rings and the category of well-mixed affine difference schemes, but only a weaker notion which we might dub temporarily
`an embedding of categories' for the lack of a reference:
the unit of the adjunction $1\to \spec\circ H^0$ is a natural isomorphism, while
the counit $1\to H^0\circ\spec$ is only a natural injection by \ref{wmaff}(\ref{sesttt}).
Also, $H^0$ is not necessarily exact.
\end{remark}

\begin{definition}
\begin{enumerate}
\item Let $(X,\Sigma)$ be a difference scheme and $(K,\varphi)$ a difference field.
A \emph{$(K,\varphi)$-rational point} of $(X,\Sigma)$ is a morphism 
$x:\spec^\varphi(K)\to(X,\Sigma)$. When $(X,\Sigma)=\spec^\Sigma(R)$, this
means we have a point $\p\in\spec^{\Sigma}(R)$ and a local map 
$(R_\p,\varphi^x)\to(K,\varphi)$, where $\varphi^x$ is the image of $\varphi$ in
$\Sigma$ by the difference structure map $()^x:\{\varphi\}\to \Sigma$.
Alternatively, we have an inclusion $(\kk(\p),\varphi^x)\hookrightarrow (K,\varphi)$.

\item Let $(X,\Sigma)$ be a difference scheme over a difference field $(k,\sigma)$
and let $(k,\sigma)\subseteq(K,\varphi)$. The set of \emph{$(K,\varphi)$-rational points} of $(X,\Sigma)$, henceforth denoted by
$(X,\Sigma)(K,\varphi)$, is the set of all $(k,\sigma)$-morphisms $\spec^\varphi(K)\to(X,\Sigma)$.
\end{enumerate}
\end{definition}

If $(R,\sigma_0)$ is a difference ring, the \emph{difference polynomial ring}
 $R\{x_1,\ldots,x_n\}=R[x_1,\ldots,x_n]_\sigma$ in $n$ variables over $(R,\sigma_0)$
is defined as the polynomial ring
$$
R[x_{1,i},\ldots,x_{n,i}:i\geq 0],
$$
together with the unique endomorphism $\sigma$ which acts as $\sigma_0$ on $R$ and
maps $x_{j,i}$ to $x_{j,i+1}$.

\begin{definition}\label{finsigmatype}
Let $(R,\sigma)$ be a difference ring.
\begin{enumerate}
\item An $(R,\sigma)$-algebra $(S,\sigma)$ is of \emph{finite $\sigma$-type} if it is an equivariant quotient of some difference polynomial ring $R[x_1,\ldots,x_n]_\sigma$. Equivalently, 
there exist $a_1,\ldots,a_n\in S$ such that $S=R[a_1,\ldots,a_n]_\sigma$.
\item An $(R,\sigma)$-difference scheme $(X,\sigma)$ is of \emph{finite $\sigma$-type} if
it is a finite union of affine difference schemes of the form $\spec^\sigma(S)$, where $(S,\sigma)$ is of finite $\sigma$-type over $(R,\sigma)$.
\item\label{finstypmor} A morphism $f:(X,\sigma)\to (Y,\sigma)$ is of finite $\sigma$-type if $Y$ is a finite union
of open affine subsets $V_i=\spec^\sigma(R_i)$ such that for each $i$, $f^{-1}(V_i)$ is
of finite $\sigma$-type over $(R_i,\sigma)$. 
\end{enumerate}
\end{definition}

\begin{remark}\label{intrmed}
The proof of \ref{wmaff}(\ref{sesttt}) in fact shows that for any difference ring $(A_1,\Sigma)$ such that
$(A,\sigma)\hookrightarrow(A_1,\sigma)\hookrightarrow(\bar{A},\sigma)$, we obtain an isomorphism of difference
schemes $\spec^\Sigma(A_1)\to \spec^\Sigma(A)$. This observation will prove invaluable 
for proving certain finiteness properties later on.
\end{remark}

For a point $x$ on a difference scheme $(X,\sigma)$, we denote by $\OO_x$ the local
(difference) ring at $x$, and by $\kk(x)$ the residue (difference) field at $x$.

\begin{definition}
Let $(K,\sigma)\subseteq (L,\sigma)$ be an extension of difference fields.
\begin{enumerate}
\item An element $\alpha\in L$ is \emph{$\sigma$-algebraic} over $K$
if the set $\{\alpha,\sigma(\alpha),\sigma^2(\alpha),\ldots\}$ is
algebraically dependent over $K$.
\item The $\sigma$-algebraic closure over $K$ defines a \emph{pregeometry} on $L$
and the dimension with respect to this pregeometry is called the $\sigma$-transcendence
degree. Alternatively, $\sigma\text{\rm-tr.deg}(L/K)$ is the supremum of numbers
$n$ such that the difference polynomial ring $K\{x_1,\ldots,x_n\}$ in $n$ variables
embeds in $L$.
\item $L$ is \emph{$\sigma$-separable} over $K$ if $L$ is linearly disjoint
from $K^\text{inv}$ over $K$, where the \emph{inversive closure} 
$(K^\text{inv},\sigma)$,  is the unique
(up to $K$-isomorphism) difference field extension of $(K,\sigma)$ where 
$\sigma$ is an automorphism of $K^\text{inv}$ and 
$$
K^\text{inv}=\bigcup_mK^{\sigma^{-m}}.
$$
\item Suppose $L$ is $\sigma$-algebraic of finite $\sigma$-type over $K$,
$\sigma$-generated by a finite set $A$. 
Let $A_k:=\bigcup_{i\leq k}\sigma^i(A)$ and let $d_k:=[K(A_k):K(A_{k-1})]$.
It is shown in \cite{cohn} that for every $k$, $d_k\geq d_{k+1}$ and
we may define the \emph{limit degree} as
$$
\dl((L,\sigma)/(K,\sigma)):=\min_k d_k.
$$
This definition is independent of the choice of the generators. 
When $L/K$ is $\sigma$-algebraic but not necessarily finitely $\sigma$-generated,
 one defines 
$\dl(L/K)$ as the maximum of $\dl(L'/K)$ where $L'$ runs over the 
extensions of finite $\sigma$-type contained in $L$.
\end{enumerate}
\end{definition}

Before introducing the various dimension/degree invariants of difference
schemes, it is useful to define an auxiliary structure where some of those
invariant will take values.
 
The ring $\N\cup\{\infty\}[\LL]$ admits a
natural lexicographic polynomial ordering $\leq$, and an equivalence relation
 $\approx$, where
$u\approx v$ if $u,v\in\N[\LL]$ have the same degree in $\LL$ and
 and their leading coefficients are equal. We will consider the
  \emph{rig} (ring without negatives) $\N\cup\{\infty\}[\LL]/{\approx}$. 

\begin{definition}\label{dimtot-ring}
Let $(k,\sigma)$ be a difference field, $(K,\sigma)$ a difference field extension and let $(R,\sigma)$ be a $(k,\sigma)$-algebra.
\begin{enumerate}
\item
Let the \emph{$\sigma$-degree} of $X$ be 
$$\dd(K/k)=\dl(K/k)\LL^{\mathop{\rm tr.deg}(K/k)}$$ in $\N\cup\{\infty\}[\LL]/{\approx}$.
\item
Let the \emph{effective $\sigma$-degree} of $X$ be
$\dde(K/k)=\dd(K^{\rm inv}/k^{\rm inv})$.
\item
Let $$\dd(R/k)=\sum\limits_{\substack{\p\in\spec(R) \\  \sigma(\p)\subseteq\p}}\dd(\kk(\p)/k),$$ and analogously for $\dde(R/k)$.
\item The  \emph{limit degree} $\dl(R/k)$ and \emph{total dimension}  $\dt(R/k)$
are defined through 
$$
\dl(R/k)\LL^{\dt(R/k)}\approx\dd(R/k),
$$ 
and analogously for the \emph{effective total dimension}.
\end{enumerate}
\end{definition}

\begin{definition}\label{dimtot-sch}
Let $(k,\sigma)$ be a difference field, and let $(X,\sigma)\to (Y,\sigma)$ of $(k,\sigma)$-difference schemes.
\begin{enumerate}
\item The \emph{$\sigma$-dimension} of $X$ is 
$
\sd(X)=\sup\limits_{x\in X}\sigma\text{\rm-tr.deg}(\kk(x)/k).
$
\item The \emph{relative $\sigma$-dimension} 
$$
\sd(\varphi)=\sup_{y\in Y}\sd(X_y),
$$
where $X_y=X\times_Y\specd(\kk(y))$ is the fibre over $y$.
\item 
Let the \emph{$\sigma$-degree} of $X$ be
$$\dd(X)=\sum_{x\in X}\dd(\OO_x/k),$$
 and analogously for $\dde(X)$.

\item The  \emph{limit degree} $\dl(X)$ and \emph{total dimension}  $\dt(X)$
are defined through 
$$
\dl(X)\LL^{\dt(X)}\approx\dd(X),
$$ 
and analogously for the \emph{effective total dimension}.

\item The \emph{relative $\sigma$-degree} 
$$
\dd(\varphi)=\sup_{y\in Y}\dd(X_y),
$$
and analogously for $\dde(\varphi)$. From these we derive the notions
of \emph{relative limit degree} and \emph{relative total dimension}.
\end{enumerate}
\end{definition} 

\begin{remark}
\begin{enumerate}
\item Clearly (cf.~\cite{udi}, \cite{laszlo}), $\sd(X)=0$ if and only $\dt(X)$ and $\dte(X)$ are
finite, and analogously for the relative dimensions. In this case, if in addition $\varphi$
is of finite $\sigma$-type, 
$\dd(\varphi)\in\N[\LL]$, i.e., the limit degree is finite.
\item
When $X=\specd(R)$, $\dt(X)=\dt(R)$, so the above definition is consistent. Indeed,
as remarked in \cite{laszlo}, the inequality
\[
\dt(R)\geq\sup_{\p\in\specd(R)}\dt(R_{\p})
\]
is obvious. In the other direction, let $\p\in\spec(R)$ such that $\sigma(\p)\subseteq\p$.
Then $\p$ induces a $\sigma$-ideal in $\spec(R_{\bar{\p}})$, where 
$\bar{\p}=\cup_{m>0}\sigma^{-m}(\p)$ is the
perfect closure of $\p$, and the opposite inequality follows.

\item When $(L,\sigma)$ is a $\sigma$-separable $\sigma$-algebraic extension of $(K,\sigma)$,
$L^{\rm inv}$ is an algebraic extension of $LK^{\rm inv}$ and 
$$\trdeg(L^{\rm inv}/K^{\rm inv})=\trdeg(LK^{\rm inv}/K^{\rm inv})=\trdeg(L/K).$$ Thus, when
$\varphi:X\to Y$ is \emph{$\sigma$-separable} in the sense that for every $x\in X$, the extension 
$\kk(x)/\kk(\varphi(x))$ is
$\sigma$-separable, we get that $$\dd(\varphi)\approx\dde(\varphi).$$


\item\label{tower} Thanks to the corresponding property of the limit degree and
the additivity of total dimension, the $\sigma$-degree is multiplicative in towers. 

\item\label{comp} Let $X=\specd(R)$. 
By the Ritt ascending chain condition
for perfect ideals in $R$ (\cite{cohn}), $X$ is a Noetherian topological space and therefore
we get a decomposition of $X$ into irreducible components,
$$
X=X_1\cup\cdots\cup X_n,
$$
where $X_i=\specd(R/\p_i)$ for some $\p_i\in \specd(R)$.
Equivalently, the zero ideal in $R$ can
be represented as
$$
0= \p_1\cap\cdots\cap\p_n.
$$
 Since $X$ is of transformal dimension $0$ (equivalently, of finite total dimension),
 for $i\neq j$, $\dt(X_i\cap X_j)<\dt(X)$ and the results of \cite{laszlo} entail
 $$
 \dd(X)\approx\sum_{i} \dd(X_i)
 \approx\sum_i\dl(\mathop{\rm Fract}(R/\p_i)/k)\LL^{\trdeg(\mathop{\rm Fract}(R/\p_i)/k)}.
 $$
An analogous statement holds for $\dde$.

\end{enumerate}

\end{remark}

\begin{proposition}[\cite{laszlo}~3.10.2]\label{dimineq}
Let $(R,\sigma)$ be a well-mixed difference algebra of finite $\sigma$-type over
a difference field $k$ and suppose $I$ is a perfect non-zero ideal.
Then $\dt(R/I)<\dt(R)$.
\end{proposition}

\section{Local vs.\ global properties}

For the proofs of the following statements up to \ref{dCRT}, we
refer the reader to \cite{ive-tgals}.
\begin{definition} 
 Let $(M,\sigma)$ be an $(A,\sigma)$-module and let $(N,\sigma)$ be a sumbodule. 
 \begin{enumerate}
 \item We say that $(M,\sigma)$ is \emph{well-mixed} if $am=0$ implies $\sigma(a)m=0$
 for all $a\in A$, $m\in M$.
 \item We say that $(N,\sigma)$ is a \emph{well-mixed submodule} of $(M,\sigma)$
 if the module $M/N$ is well-mixed.
 \end{enumerate}
\end{definition} 

Clearly $(M,\sigma)$ is well-mixed if and only if the annihilator $\mathop{Ann}(m)$ of
any $m\in M$ is a well-mixed $\sigma$-ideal in $(A,\sigma)$. Indeed, if $ab\in \mathop{Ann}(m)$,
then $a(bm)=0$ so $\sigma(a)(bm)=(\sigma(a)b)m=0$ so $\sigma(a)b\in\mathop{Ann}(m)$.

Moreover, since the intersection of well-mixed submodules is well-mixed and $M$ is trivially a well-mixed submodule of itself, for every submodule $(N,\sigma)$ of $(M,\sigma)$ there exists  a smallest well-mixed submodule $[N]_w$ containing $N$. Thus $[0]_w$ is the smallest 
well-mixed submodule of $(M,\sigma)$ associated with
the largest well-mixed quotient $M_w$ of $M$.

\begin{proposition}\label{loczero}
Let $(M,\sigma)$ be a well-mixed $(A,\sigma)$-module.
The following are equivalent.
\begin{enumerate}
\item\label{uno} $M=0$;
\item\label{due} $M_\p=0$ for every $\p\in\spec^\sigma(A)$.
\item\label{tre}  $M_\p=0$ for every $\p$  maximal in $\spec^\sigma(A)$.
\end{enumerate}
\end{proposition}

\begin{corollary}\label{wkloczero}
Let $(M,\sigma)$ be an $(A,\sigma)$-module. If $(M_\p)_w=0$ for every $\p$ 
maximal in $\spec^\sigma(A)$, then $M_w=0$.
\end{corollary}

\begin{proposition}\label{locinj}
Let $\phi:(M,\sigma)\to (N,\sigma)$ be an $(A,\sigma)$-module homomorphism and assume
that $(M,\sigma)$ is well-mixed. The following are equivalent.
\begin{enumerate}
\item\label{unoo} $\phi$ is injective;
\item\label{duee} $\phi_\p:M_\p\to N_\p$ is injective for every $\p\in\spec^\sigma(A)$.
\item\label{tree} $\phi_\p:M_\p\to N_\p$ is injective for every $\p$ maximal in $\spec^\sigma(A)$.
\end{enumerate}
\end{proposition}

\begin{proposition}\label{locwksurj}
Let $\phi:(M,\sigma)\to (N,\sigma)$ be an $(A,\sigma)$-module homomorphism. 
If $\phi_\p:M_\p\to N_\p$ is almost surjective for every $\p$ maximal in $\spec^\sigma(A)$, then
$\phi$ is almost surjective, $[\mathop{\rm im}(\phi)]_w=N$ (equivalently, $\mathop{\rm coker}(\phi)_w=0$).
\end{proposition}

\begin{proof}
Let $N'=\mathop{\rm coker}(\phi)$. Then $M\to N\to N'\to0$ is exact, and by localisation 
$M_\p\to N_\p\to N'_p\to 0$ is exact for every $\p\in\spec^\sigma(A)$. By assumption, 
$(N'_\p)_w=0$ for all $\p$ maximal in $\spec^\sigma(A)$ and \ref{wkloczero} implies that $N'_w=0$.
\end{proof}

\begin{proposition}\label{tensorwm}
Let $(M,\sigma)$ and $(N,\sigma)$ be $(A,\sigma)$-modules with $(N,\sigma)$ well-mixed.
Then $(M,\sigma)\otimes_{(A,\sigma)}(N,\sigma)$ is well-mixed.
\end{proposition}

\begin{proposition}\label{locflat}
Let $(M,\sigma)$ be a well-mixed $(A,\sigma)$-module. The following are equivalent.
\begin{enumerate}
\item\label{unooo} $M$ is a flat $A$-module.
\item\label{dueee} $M_\p$ is a flat $A_\p$-module for every $\p\in\spec^\sigma(A)$.
\item\label{treee} $M_\p$ is a flat $A_\p$-module for every $\p$ maximal in $\spec^\sigma(A)$.
\end{enumerate}
\end{proposition}

\begin{remark}
Let $(A,\sigma)\to(B,\sigma)$ be a homomorphism of well-mixed difference rings such that
$B$ is a flat $A$-module and denote by $\bar{A}$ and $\bar{B}$ the rings of global sections
of $\spec^\sigma(A)$ and $\spec^\sigma(B)$. We can consider $\bar{B}$ as an $A$-module
via the morphism $A\hookrightarrow\bar{A}\to\bar{B}$ as in \ref{embedcat}, and
we can conclude that $\bar{B}$ is flat over $A$. 
\end{remark}

\begin{proposition}
Let $(A,\sigma)$ be a well-mixed domain. If $A_\p$ is normal for every $\p$
maximal in $\spec^\sigma(A)$,
then $A$ is almost normal.
\end{proposition}

\begin{proposition}[Difference Chinese Remainder Theorem]\label{dCRT}
Let $(R,\sigma)$ be a difference ring and let $\p_1,\ldots,\p_n$ be 
pairwise weakly separated difference ideals, i.e.\ $\{\p_i+\p_j\}=R$ for $i\neq j$.
The natural morphism
$$
R\to R/\p_1\times\cdots\times R/\p_n
$$
is almost surjective, with kernel $\cap_i\p_i$.
\end{proposition}
\begin{proof}

From the maximality of the $\p_i$, no $\p_i$ is contained in $\p_j$ for $i\neq j$.
Let us consider the above difference ring morphism as a morphism $\varphi$ of
$(R,\sigma)$-modules. 
Let $\p$ be a maximal element of $\spec^\sigma(R)$. The condition of
weak separatedness implies that $\p$ contains at most one $\p_i$.
If $\p_i\subseteq\p$, then $(R/\p_i)_\p\simeq R_\p/\p_iR_\p$, and if
$\p_i\not\subseteq\p$, $(R/\p_i)_\p\simeq 0$. 
Thus, localising $\varphi$ at $\p$ containing some $\p_i$ yields the natural morphism
$\varphi_{\p}:R_{\p}\to R_{\p}/\p_iR_{\p}$, which is surjective. Localising
at a maximal element $\p\in\spec^\sigma(R)$ not containing any of the $\p_i$
yields $\varphi_\p:R_\p\to 0$, which is again surjective. Therefore, $\varphi_\p$
is surjective for every maximal $\p\in\spec^\sigma(R)$ and we deduce by 
\ref{locwksurj} that $\varphi$ is almost surjective.
\end{proof}
\begin{corollary}

Let $(R,\sigma)$ be a well-mixed difference ring and suppose that 
$\spec^\sigma(R)$ is finite and consists only of maximal elements
 $\{\p_1,\ldots,\p_n\}$, say.
Then the natural morphism
$$
R\mapsto R_{\p_1}\times\cdots\times R_{\p_n},
$$
mapping $r\mapsto(r/1,\ldots,r/1)$,
is injective and almost surjective.
\end{corollary}
\begin{proof}
Let $\q_i=\ker(R\to R_{\p_i})=\{x:\text{ there exists an }r\in R\setminus\p_i,\ rx=0\}$,
so that $R/\q_i\hookrightarrow R_{\p_i}$.
Since $\spec^\sigma R_{\p_i}=\{0,\p_i R_{\p_i}\}$, $\p_iR_{\p_i}$ consists of
$\sigma$-nilpotent elements, so $\p_i/\q_i$ consists only of $\sigma$-nilpotent
elements and we conclude that $\{\q_i\}=\p_i$ and by maximality, $\p_i$ is the
only element of $\spec^\sigma(R)$ that contains $\q_i$.
Thus no $\p_i$ contains both $\q_i$ and $\q_j$ for $i\neq j$ so $\{\q_i+\q_j\}=R$ for $i\neq j$, i.e. the $\q_i$ are pairwise weakly separated.
 
 We claim that $\cap_i\q_i=0$.
 To this end, pick an $x\in\cap_i\q_i$. By the definition of $\q_i$, this means
 that for every $i$ there is an $r_i\in R\setminus\p_i$ such that $r_ix=0$.
 Since $R$ is well-mixed, we get that $\sigma^l(r_i)x=0$ for all $l$.
 Thus, letting $I=\langle r_1,\ldots,r_n\rangle_\sigma$, we have that $Ix=0$.
 But $I\not\subseteq\p_i$ for any $i$, so $I=R$ and $1\cdot x=0$ so $x=0$.
 
 By the Difference Chinese Remainder Theorem~\ref{dCRT}, 
 we get that $R\to R/\q_1\times\cdots\times R/\q_n$ is injective and
 almost surjective, and it remains to show that $R/\q_i\hookrightarrow R_{\p_i}$
 is bijective. However, by the above discussion,  $R/\q_i$ is a local ring with 
 maximal ideal $\bar{\p}_i=\p_i/\q_i$ so
 $R/\q_i\simeq(R/q_i)_{\bar{\p}_i}\simeq R_{\p_i}/\q_iR_{\p_i}\simeq R_{\p_i}$.
 \end{proof}

\section{Difference curves}

\begin{definition} An \emph{affine difference curve} is a difference scheme of the form $(X,\Sigma)=\spec^\Sigma(R)$
where
$(R,\Sigma)$ is an algebra of finite $\sigma$-type over a difference field
$(k,\sigma)$ which is integral and of total dimension 1.
\end{definition}

\begin{remark}\label{remht1}
It is immediate from \ref{dimineq} that the set of 
closed points of $(X,\sigma)$ corresponds to the set of $\sigma$-height one ideals in $\spec^\sigma(R)$.
\end{remark}

\section{Nakayama with a difference}

\begin{lemma}\label{lemma:jacobson-skew}
Let $(k,\sigma)$ be a difference field. The Jacobson ideal of the
skew polynomial ring $k[x;\sigma]$ is zero.
\end{lemma}
\begin{proof}
By \cite{bedi-ram}, 
$$
J(k[x;\sigma])=\{\sum_i\alpha_ix^i:\alpha_0\in I\cap J(k), \alpha_i\in I, i\geq 1\}=
\{\sum_{i\geq 1}\alpha_ix^i:\alpha_i\in I\},
$$
where $I=\{\alpha\in k:\alpha x\in J(k[x;\sigma])\}$.
Thus, $\alpha\in I$ implies that $\alpha x$ belongs to every maximal ideal. 
In particular, the ideal $\langle x-1\rangle$ being maximal by \cite{keating},
Exercise~3.2.1, we get that $\alpha x\in\langle x-1\rangle$, which implies $\alpha=0$.
\end{proof} 

\begin{proposition}\label{prop:jacobson-skew}
Let $(R,\M,\sigma)$ be a local difference ring. The Jacobson radical of the
skew polynomial ring $R[x;\sigma]$ consists of the twisted polynomials with
coefficients in $\M$, i.e.,
$$
J(R[x;\sigma])=\M[x;\sigma].
$$
\end{proposition}
\begin{proof}
Let us denote by $\pi_0:R\to k=R/\M$ the residue map, and consider the morphism
$\pi:R[x;\sigma]\to k[x;\sigma]$ defined by $\pi(\sum_ia_ix^i)=\sum_i\pi_0(a_i)x^i$.

Since $\pi$ is surjective,  $\pi(J(R[x;\sigma]))\subseteq J(k[x;\sigma])$.
By \ref{lemma:jacobson-skew},  $J(k[x;\sigma])=0$ and therefore
$J(R[x;\sigma])\subseteq\pi^{-1}(0)=\M[x;\sigma]$.

On the other hand, let $I$ be a maximal ideal in $R[x;\sigma]$. 
Then $I\subseteq \pi^{-1}(\pi(I))$ so by the maximality of $I$, either
$I=\pi^{-1}(\pi(I))$ or $\pi^{-1}(\pi(I)=R$, but the latter is impossible since $\pi$ is onto.

Thus, for every maximal $I$, $\M[x;\sigma]=\pi^{-1}(0)\subseteq\pi^{-1}(\pi(I))=I$
so $J(R[x;\sigma])\supseteq\M[x;\sigma]$ as well.
\end{proof}

\begin{proposition}[Nakayama's Lemma/Jacobson-Azumaya's Theorem--classical version]
Let $R$ be a unitary ring and
let $M$ be a finitely generated left $R$-module, and let $I$ be an ideal contained
in the Jacobson radical of $R$, $I\subseteq J(R)$.
Then $IM=M$ implies that $M=0$.
\end{proposition}

\begin{proposition}[Nakayama's Lemma--skew version]
Let $(R,\M,\sigma)$ be a local difference ring, and let $M$ be a finitely generated
$R[x;\sigma]$-module. Then $\M[x;\sigma]M=M$ implies $M=0$.
\end{proposition}
\begin{proof}
This is an immediate consequence of the classical version, using the
fact that $J(R[x;\sigma])=\M[x;\sigma]$ established in \ref{prop:jacobson-skew}.
\end{proof}

\begin{proposition}[Nakayama's Lemma--difference version]
Let $(R,\M,\sigma)$ be a local difference ring and let $(M,\sigma)$ be
a finitely $\sigma$-generated $(R,\sigma)$-module. If $\M M=M$, then $M=0$.
\end{proposition}
\begin{proof}
We can consider $M$ as a natural left $R[x;\sigma]$-module and the condition
that $(M,\sigma)$ is finitely $\sigma$-generated as an $(R,\sigma)$-module
implies that $M$ is a finitely generated $R[s;\sigma]$-module. Given that
$\M M=\M[x;\sigma]M$, it suffices to apply the skew version.
\end{proof}

\begin{corollary}
Let $(M,\sigma)$ be a finitely $\sigma$-generated $(R,\M,\sigma)$-module,
and let $(N,\sigma)$ be a $(R,\sigma)$-submodule. If $M=\M M+N$, then
$M=N$.
\end{corollary}
\begin{proof}
Apply the difference version of Nakayama to $M/N$, observing that 
$\M(M/N)=(\M M+N)/N$.
\end{proof}

\begin{proposition}
Let $(M,\sigma)$ be finitely $\sigma$-generated $(R,\M,\sigma)$-module
and assume $x_1,\ldots,x_n\in M$ are such that their images in $M/\M M$
$\sigma$-span this $(k,\sigma)=(R/\M,\sigma)$-vector space. Then the
$x_i$ $\sigma$-generate $M$.
\end{proposition}
\begin{proof}
Let $N$ be the submodule of $M$ $\sigma$-generated by the $x_i$.
Then $N\to M\to M/\M M$ maps $N$ onto $M/\M M$ and $N+\M M=M$
so $N=M$ be the Corollary.
\end{proof}

\begin{remark}
It may not be possible to choose $x_1,\ldots,x_n$ such that they $\sigma$-freely
span $M/\M M$.
\end{remark}

\section{Non-singularity}

\begin{proposition}
Let $(R,\M,\sigma)$ be a local difference ring of total dimension 1 with residue field $k$. The following statements are equivalent:
\begin{enumerate}
\item $\M$ is $\sigma$-principal, i.e., it is $\sigma$-generated by a single element;
\item the $(k,\sigma)$-vector space $\M/\M^2$ is $\sigma$-generated by a single element.
\end{enumerate}
\end{proposition}

\begin{proposition}
Let $(R,\M,\sigma)$ be a local difference ring of total dimension 1 with residue field $k$. If the
(ordinary) vector space dimension of the $k$-vector space $\M/\M^2$ is 1,
then the completion $\hat{R}=\varprojlim_iR/\M^i$ of $R$ is a discrete valuation
ring and $\hat{\M}$ is principal.
\end{proposition}

However, in the above proposition, it does not follow that $R$ itself is a 
discrete valuation ring. For our purposes, the notions of non-singularity
that might be extrapolated from the above propositions
are not sufficiently well-behaved, so we choose to work with the following stronger notion.

\begin{definition}
Let $X$ be a difference curve and let $x\in X$  be a closed point.  We say that
$x$ is \emph{non-singular} or \emph{regular} if the difference local ring 
$(\OO_x,\M_x,\sigma)$ is
a regular local domain of dimension one in the usual sense, i.e., a discrete valuation ring. We say that $X$ is \emph{non-singular} or \emph{regular} if it is so at every point
$x\in X$.
\end{definition}

Given the similarity of this definition with the definition of non-singularity in
classical algebraic geometry, one might ask whether it is reasonable to expect
to have any non-singular points whatsoever, especially in view of the fact that
the underlying ring $R$ of $X$ is typically not Noetherian. The following result
(which is a special case of more general consideration from \cite{ive-tgals})
puts one's mind at rest.

\begin{proposition}[Generic non-singularity]
Let $X$ be a difference curve over a difference field of characteristic 0. 
There is an nonempty open subset $U$ of $X$ such that every point of $U$
is non-singular.
\end{proposition}

\begin{proof}
Suppose $X=\spec^\sigma(R)$ with $R$ finitely $\sigma$-generated domain over a
difference field $k$, i.e.\ there exists a finite tuple $\bar{a}=a_1,\ldots,a_n$ such that
$R=k[\bar{a}]_\sigma=k[\bar{a},\sigma\bar{a},\ldots]$. Let $K$ be the fraction
field of $R$, and consider $R_i=k[\bar{a},\sigma\bar{a},\ldots,\sigma^i\bar{a}]$
and the corresponding field of fractions $K_i$. From the theory of limit degree,
we know that there exist $N$ and $d$ such that for $i\geq N$, $[K_{i+1}:K_i]=d$.
By taking a larger set of generators, we may assume that $N=0$, i.e.,
for all $i\geq0$, $[K_{i+1}:K_i]=d$. Since $K_1$ is separable over $K_0$,
there exists an $\alpha\in K_1$ such that $K_0(\sigma\bar{a})=K_1=K_0(\alpha)$.
There is no harm in assuming that $\alpha$ is integral over $R_0$ with
minimal polynomial $f$ over $R_0$. Since $\alpha$ generates $K_1$, we
have
$$
\sigma a_i=\sum_{j}\gamma_{ij}\alpha^j,
$$
 for some $\gamma_{ij}\in K_0$.
By $\sigma$-localising with denominators of $\gamma_{ij}$, we may assume that
$R_1\subseteq R_0[\alpha]$. %
By $\sigma$-localising further, we may assume that all $\sigma a_i$ are integral over $R_0$. Using $K_1=K_0(\alpha)$, it follows that $\alpha$ is a $K_0$-linear
combination of (bounded) powers of $\sigma a_i$ and we need another 
$\sigma$-localisation to finally conclude
that $R_1=R_0[\alpha]$, thus making $R_1$ into a finite free $R_0$-module.
This in turn implies that $R_{i+1}=R_i[\sigma^i\alpha]$
for all $i$, and the minimal polynomial of $\sigma^i\alpha$ over $K_i$ is 
$f^{\sigma^i}$. 

By generic non-singularity, we can localise to make $R_0$ a regular ring,
and by generic \'etaleness (or by just $\sigma$-localising by $f'$) we can assume
that $R_1$ is \'etale over $R_0$, i.e., that the formal derivative $f'$ is invertible
in $R_1$. This entails that each $(f^{\sigma^i})'=(f')^{\sigma^i}$ is invertible in $R_{i+1}$, which means that each $R_{i+1}$ is \'etale over $R_i$. We are now in a
situation where, given an $x\in\spec^\sigma(R)$, and writing $x_i$ for the
corresponding projection in $\spec(R_i)$, the local ring $\OO_{X,x}$ is a discrete
valued ring, being the direct limit
of the unramified system of discrete valuation rings $\OO_{\spec(R_i),x_i}$.
\end{proof}

\begin{remark}
The same result holds in positive characteristic if $X$ has enough
separability built in, for example if the reduced limit degree of $X$ equals its limit
degree.
\end{remark}

\begin{lemma}\label{Lsigma}
Let $X$ be a non-singular difference curve and let $x\in X$ be a non-singular point. 
\begin{enumerate}
\item For any  $f\in \OO_x$, $\sigma(f)/f\in \OO_x$.
\item For any $f\in\OO_X(X)$, $\sigma(f)/f\in\OO_X(X)$.
\end{enumerate}
\end{lemma}
\begin{proof}
Statement (2) is immediate from (1). To see (1), if $t$ is the uniformiser at $x$
(i.e.\ the generator of $\M_x$), it suffices to check that  
$\sigma(t)/t\in\OO_x$. However, since $\sigma(\M_x)\subseteq\M_x$,
$\sigma(t)$ must be divisible by $t$.
\end{proof}

\begin{definition}\label{defkrull}
Let $(R,\sigma)$ be a difference domain, let $X=\spec^\sigma(R)$, and denote by
$X^{(1)}=\spec^{\sigma,(1)}(R)=\{\p\in\spec^\sigma(R): \mathop{\rm ht}(\p)=1\}$ the set of height one
$\sigma$-prime ideals. We say that $(R,\sigma)$ is a \emph{$\sigma$-Krull domain} if 
\begin{enumerate}
\item[(DVR)]  for each $\p\in X^{(1)}$, the localisation $R_\p$ is a discrete valuation ring;
\item[(INT)] considering $\tilde{R}=\cap_{\p\in X^{(1)}}R_\p$ as an $(R,\sigma)$-module,
$[R]_w=\tilde{R}$;
\item[(FC)] each $f\in R\setminus\{0\}$ is contained in only finitely many $\p$ from $X^{(1)}$.
\end{enumerate}
\end{definition}

\begin{remark}\label{rem2ht1}
Let $X=\spec^\sigma(R)$ be a non-singular difference curve. 
Given that for each $\p\in\spec^\sigma(R)$, $R_\p$ is a discrete valuation ring, 
we have that each non-zero $\p$ must
be of (ordinary) height 1. In view of \ref{remht1}, we conclude that the
set of closed points of $X$ corresponds to the set of height one elements
of $\spec^\sigma(R)$, which equals the set of $\sigma$-height one elements
of $\spec^\sigma(R)$. Thus, in the case of non-singular curves, we can use the notation $\spec^{\sigma,(1)}(R)$
for the latter two sets without confusion.
\end{remark}

\begin{proposition}\label{propkrull}
Let $X=\spec^\sigma(R)$ be a non-singular difference curve. We have the following.
\begin{enumerate}
\item\label{jen} $(R,\sigma)$
is a $\sigma$-Krull domain and $\tilde{R}=\bar{R}\cap K$.
\item\label{dva} $\tilde{R}$ is a Krull domain (in the usual sense) and
$\spec^{(1)}(\tilde{R})=X^0$.
\end{enumerate}
\end{proposition}
\begin{proof}
\noindent(\ref{jen}) Firstly we note that by \ref{rem2ht1}, $X^{(1)}$ is the
set of all closed points, so the property (DVR) follows from the definition of non-singularity. For (FC), let us take some nonzero $f\in R$. By the Ritt property
(\cite{laszlo}~3.3.2), we have a unique irredundant decomposition 
$V^\sigma(f)=V^\sigma(\p_1)\cup\cdots\cup V^\sigma(\p_n)$ with 
$\p_i\in\spec^\sigma(R)$. Equivalently, $\{f\}_\sigma=\p_1\cap\cdots\cap\p_n$, 
so $p_i$ are clearly the only ideals of $\sigma$-height 1 containing $f$.
By \ref{rem2ht1}, $\tilde{R}=\cap_{\p\in \spec^\sigma(R)}R_\p$, and consider
the inclusion $R\to\tilde{R}$ as a morphism of $(R,\sigma)$-modules.
Then, for each $\p\in\spec^\sigma(R)$, its localisation $R_\p\to\tilde{R}_\p=R_\p$
is onto, so we conclude that $[R]_w=\tilde{R}$, thus establishing the property (INT).
Let us write $\bar{R}$ for the global sections ring $\OO_X(X)$. By \ref{Lsigma}, for every $f\in R$, $\sigma(f)/f\in\bar{R}$, so $\tilde{R}=[R]_w\subseteq\bar{R}$.
Conversely, if $a/b\in\bar{R}$, by the definition of global sections, for each $\p$
we have that $a/b\in R_\p$, so $a/b\in\tilde{R}$.

\noindent(\ref{dva}) 
Since $\tilde{R}$ is an intersection of discrete valuation rings, in order to
show that it is Krull, it suffices to show that it has the finite character property, i.e.,
that every $\bar{s}\in\tilde{R}$ is not a unit in only finitely many of those.

But we have even more, i.e., every $\bar{s}\in\bar{R}$ is only contained in
finitely many elements of $\spec^\sigma(\bar{R})\simeq\spec^\sigma(R)$.
Indeed, by \ref{wmaff}, for every $\p_0\in\spec^\sigma(R)$ there exist $g,a\in R$ with
$g\notin\p$ such that $g\bar{s}=a$ on 
$\spec^\sigma(R)$. Thus, 
for every $\p\in D^\sigma(g)$, $\bar{s}\in\bar{\p}$ if and only if $a\in\p$. Thus,
every $\p_0\in\spec^\sigma(R)$ has a neighbourhood in which $\bar{s}$ has
only finitely many `zeroes'. By quasi-compactness of $\spec^\sigma(R)$, it
follows that $\bar{s}$ has only finitely many `zeroes'  altogether. 

Denote $Y=\spec^\sigma(\tilde{R})^0=\spec^{\sigma,(1)}(\tilde{R})\subseteq
\spec^{(1)}(\tilde{R})$. By \ref{intrmed}, $Y\simeq X^0=X^{(1)}$,
so  $\cap_{\p\in Y}\tilde{R}_\p=\cap_{\p\in X^{0}}R_\p=\tilde{R}$ again,
we see that $\tilde{R}$ is a subintersection and \cite{fossum}, Proposition~3.15,
implies that $\spec^{(1)}(\tilde{R})=Y\simeq X^0$.
\end{proof}

\begin{lemma}[\cite{matsumura-comm-ring-th}, Exercise~1.6]\label{exercis}
Let $A$ be a ring, $I,P_1,\ldots,P_r$ ideals of $A$ and suppose $P_3,\ldots,P_r$
prime and $I$ is not contained in any of the $P_i$. Then there exists an $x\in I$
not contained in any $P_i$.
\end{lemma}

\begin{proposition}[Approximation theorem]\label{prop:approx}
Let $(R,\sigma)$ be a $\sigma$-Krull domain with fraction field $K$.
Let $X=\spec^\sigma(R)$ and  let $X^{(1)}=\spec^{\sigma,(1)}(R)$
 be the set of elements of $X$ of height 1. Suppose we are given
 $\p_1,\ldots,\p_r\in X^{(1)}$ and $e_1,\ldots, e_r\in\Z$. Then there exits
 an $f\in K$ such that $v_{\p_i}(f)=e_i$ for $1\leq i\leq r$ and $v_{\p}(f)\geq 0$
 for all $\p\in X^{(1)}\setminus\{\p_1,\ldots,\p_n\}$.
\end{proposition}
\begin{proof}
Since $\p_i$ are of height 1, there are no inclusions between them.
Since each $R_{\p_i}$ is a discrete valuation ring, 
$\p_i\not\subseteq\p_i^{(2)}=\p_i^2R_{\p_i}\cap A$.
By \ref{exercis}, there exist 
$g_i\in\p_i\setminus\left(\p_i^{(2)}\cup\bigcup_{j\neq i}\p_j\right)$.
Then $v_{\p_j}(g_i)=\delta_{ij}$ and we set $g=\prod_{i=1}^r g_i^{e_i}$.
Let $\p_1',\ldots,\p_s'$ be all the elements $\p$ 
of $X^{(1)}\setminus\{\p_1,\ldots,\p_r\}$
such that $v_\p(g)<0$. Then choosing for each $j=1,\ldots,s$ an element
$t_j\in \p_j'\setminus\bigcup_i\p_i$, we see that
$f=g(t_1,\ldots,t_s)^l$ satisfies the requirements of the theorem for sufficiently
large $l$.
\end{proof}

\section{Multiplicities, divisors}

\begin{definition}
Let $X=\spec^\sigma(R)$ be a non-singular difference curve over a difference
field $k$ and let $X^0=X^{(1)}$
be the set of closed (height one) points.
A \emph{prime divisor} on $X$ is just an element of $X^0$.
A \emph{Weil divisor} is an element of the free abelian group $\Div X$
generated by $X^{(1)}$. A divisor $D=\sum_i n_i x_i$ is \emph{effective} if
all $n_i\geq 0$.
\end{definition}

\begin{definition}
Let $X$ be as above and let $K$ be its function field (the fraction field of $R$).
For $f\in K^{\times}$, we let the \emph{divisor} $(f)$ of $f$ on $X$ be
$$
(f)=\sum_{x\in X^{(1)}}v_x(f)\cdot x.
$$
By \ref{propkrull}, this is a divisor. Any divisor which is equal to the divisor
of a function is called a \emph{principal} divisor.
\end{definition}

Note that $(f/g)=(f)-(g)$ so $f\mapsto (f)$ is a homomorphism
$K^{\times}\to\Div X$ whose image is the subgroup of principal divisors.

\begin{definition}
Let $X$ be a non-singular difference curve over $k$. Two divisors $D,D'\in\Div X$
are \emph{linearly equivalent}, written $D\sim D'$, if $D-D'$ is a principal divisor.
The \emph{divisor class group} $\Cl X$ is the quotient of $\Div X$ by the
subgroup of principal divisors.
\end{definition}

\begin{remark}
There exists a well-developed theory of divisors on Krull domains, cf.~\cite{fossum}.
In view of \ref{propkrull}, comparing the definitions, we see that the
group of divisors (resp.\ the divisor class group) of a non-singular difference
curve $X=\spec^\sigma(R)$ is nothing other than the group of divisors
(resp.\ the divisor class group) of the Krull domain $\tilde{R}$ associated with it.
The general theory shows that, in this non-singular case, the theory of
Weil divisors coincides with the theory of \emph{Cartier divisors} and
\emph{invertible sheaves}, but we shall not need these in the present paper.
\end{remark}

\begin{definition}
For a divisor $D=\sum_in_ix_i$, we define the \emph{degree} of $D$
as $\deg(D)=\sum_in_i\cdot\dl(\kk(x_i)/k)$, making $\deg$ into a homomorphism $\Div X\to\Z$.
\end{definition}


\section{Ramification and preservation of multiplicity}

\begin{definition}
Let $(B,\sigma)\to (A,\sigma)$ be an extension of difference rings. We say that
$(A,\sigma)$ is \emph{$\sigma$-finite} over $(B,\sigma)$ if $B$ is integral of
finite $\sigma$-type over $A$. It is equivalent to say that there exists a finite
tuple $a=a_1,\ldots,a_n\in A$, such that, writing 
$A_i=B[a,\sigma a,\ldots,\sigma^{i-1}a]$, $A_{i+1}$ is a finite $A_i$ module for
every $i\geq 0$ and $A=\cup_i A_i$.
\end{definition}

\begin{definition}
Suppose we have a 
morphism $\pi:(X,\Sigma)\to(Y,\sigma)$,
and we pick a point $y\in Y$ and $x\in X$ with $\pi(x)=y$.
The \emph{ramification index} of $\pi$ at $x$ is defined as 
$$
e_x(\pi)=v_x(\pi^\sharp t_y),
$$
where $\pi^\sharp$ is the local morphism $\OO_y\to\OO_x$ induced by $\pi$
and $t_y$ is a uniformiser at $y$, i.e., a generator of the maximal ideal $\M_y$.

When $\pi$ is $\sigma$-finite, we can define a morphism $\pi^*:\Div Y\to\Div X$
by extending the rule
$$
\pi^*(y)=\sum_{\pi(x)=y}e_x(\pi)\cdot x
$$
for prime divisors $y\in Y$ by linearity to $\Div Y$.
\end{definition}

\begin{lemma}\label{lem:sigmafin}
Let $\varphi:(X,\sigma)\to(Y,\sigma)$ be a $\sigma$-finite morphism of non-singular difference curves. For $y\in Y$, assume $\varphi^{-1}(y)=\{x_1,\ldots,x_r\}\neq\emptyset$
and let $\tilde{\OO}=\cap_i\OO_{x_i}$.
Then $\tilde{\OO}$ is almost $\sigma$-finite over $\OO_y$.
\end{lemma}
\begin{proof}
We may assume that $X$ and $Y$ are affine, $X=\spec^\sigma(A)$ and
$Y=\spec^\sigma(B)$. Then we have $(\kk(Y),\sigma)\hookrightarrow (\kk(X),\sigma)$,
which allows us to consider $\kk(Y)$ as a subfield of $\kk(X)$. Moreover,
$A$ is $\sigma$-finite over $B$.
Let us prove that $\tilde{\OO}=\tilde{A}\OO_y$. If $f\in\tilde{\OO}$ and $z_i$ are the poles
of $f$ on $X$, then $y_i=\varphi(z_i)\neq y$. By the Approximation 
Theorem~\ref{prop:approx}, there exists a function $h$ such that $h(y)\neq 0$,
$h(y_i)=0$ and $fh\in\OO_{z_i}$. Thus, $fh\in \tilde{A}$, and since $h^{-1}\in\OO_y$,
we get that $f\in\tilde{A}\OO_y$. This establishes that $\tilde{\OO}\subseteq\tilde{A}\OO_y$, and the converse inclusion is obvious.

Now, since $A$ is $\sigma$-finite over $B$, 
we get that $A\OO_y$ is $\sigma$-finite over $\OO_y$, so 
$\tilde{A}\OO_y$ is almost $\sigma$-finite over $\OO_y$.
\end{proof}

\begin{proposition}\label{sifree}
Let $\varphi:(X,\sigma)\to (Y,\sigma)$ be a strongly $\sigma$-finite morphism of non-singular difference curves. 
For $y\in Y$, assume $\varphi^{-1}(y)=\{x_1,\ldots,x_r\}\neq\emptyset$
and let $\tilde{\OO}=\cap_i\OO_{x_i}$.
Then $\tilde{\OO}$ is $\sigma$-free over $\OO_y$ of limit rank $\dl(X/Y)$.
\end{proposition}
\begin{proof}
Since the statement is local, we can reduce to the case where $\varphi$
is $\spec^\sigma$ of some morphism $(B,\sigma)\to (A,\sigma)$ so that
$\tilde{A}$ is $\sigma$-finite over $\tilde{B}$. Thus $\tilde{A}$ is a direct limit
of some $\tilde{A}_i$ such that $\tilde{A}_{i+1}$ is finite over $\tilde{A}_i$ for $i\geq0$
and $A_0=\tilde{B}$. By \ref{lem:sigmafin},
$\tilde{\OO}=\tilde{A}\OO_y$, which is then $\sigma$-finite over $\OO_y$.
In other words, we can write $\tilde{\OO}$ as the direct limit of $\tilde{\OO}_i$
such that each $\tilde{\OO}_{i+1}$ is finite over $\tilde{\OO}_i$ for $i\geq 0$ and
$\tilde{\OO}_0=\OO_y$. At the same time, we can arrange
that $\tilde{\OO}_i=\cap_j\OO_{\pi_i(x_j)}$ where $\pi_i(x_j)$ is the projection
of $x_j$ to the $i$-th component $X_i$ of the system of prolongations corresponding
to the $\tilde{A}_i$ above. By non-singularity, each $\OO_{\pi_i(x_j)}$ is
a discrete valuation ring 
and thus each $\tilde{\OO}_i$ is in particular a Pr\"ufer domain, being a finite intersection of discrete valuation rings. Using the main structure theorem for
modules over Pr\"ufer domains, since each $\tilde{\OO}_{i+1}$ is a finite module over
$\tilde{\OO}_i$, it decomposes into a free part and a torsion part. Since both
$\tilde{\OO}_{i+1}$ and $\tilde{\OO}_i$ are contained in the field $\kk(X)$, the
torsion part is trivial and we conclude that each $\tilde{\OO}_{i+1}$ is
a free $\tilde{\OO}_i$-module, of rank $r_i$, say. It remains to show that
$r_i=d_i$ where $d_i=[\kk(X_{i+1}):\kk(X_i)]$. The rank $r_i$ is the maximal
size of a subset of $\tilde{\OO}_{i+1}$ which is linearly independent over  
$\tilde{\OO}_i$, or, equivalently, over $\kk(X_i)$. Clearly, $r_i\leq d_i$,
so we need to find $d_i$ elements of $\tilde{\OO}_{i+1}$ constituting
a set which is linearly independent  over $\kk(X_i)$. Starting with a basis
$g_1,\ldots,g_{d_i}$ of $\kk(X_{i+1})$ over $\kk(X_i)$, let $e$ be greater than
the order of poles of any $g_j$ at $\pi_{i+1}(x_l)$. 
Let $t_y$ be the uniformiser at $y$ and let $f$ denote the image of $t_y^e$ in
$K(X_0)\subseteq K(X_i)$.  
We have that $fg_j$ is regular at every $\pi_{i+1}(x_l)$.
Thus $\{fg_1,\ldots,fg_{d_i}\}$ is contained in $\tilde{\OO}_{i+1}$ and yet it is
still linearly independent over $\kk(X_i)$.
\end{proof}

\begin{theorem}\label{degdl}
Let $\varphi:(X,\sigma)\to (Y,\sigma)$ be as in \ref{sifree}. Then, for every $y\in Y$,
if $\varphi^{-1}(y)=\{x_1,\ldots,x_r\}\neq\emptyset$, then
$$
\deg(\varphi^*(y))=\dl(X/Y).
$$
\end{theorem}
\begin{proof}
By \ref{sifree}, we know that
$\tilde{\OO}=\cap_i\OO_{x_i}$ is $\sigma$-free over $\OO_y$ of limit rank $\dl(X/Y)$.
Let $t=t_y$ be a uniformiser at $y$, $\M_y=(t_y)$.
Then $t\tilde{\OO}=\cap_i(t\OO_{x_i}\cap\tilde{\OO})$, so by the Difference
Chinese Remainder Theorem \ref{dCRT}, 
$$
\tilde{\OO}/{t\tilde{\OO}}
\approx
\tilde{\OO}/{t\OO_{x_1}\cap\tilde{\OO}}\times\cdots\times
\tilde{\OO}/{t\OO_{x_r}\cap\tilde{\OO}}\simeq
\OO_{x_1}/{t\OO_{x_1}}\times\cdots\times\OO_{x_r}/{t\OO_{x_r}}.
$$
Note that by \ref{Lsigma}, the right-hand side
is automatically well-mixed as an $\tilde{\OO}$-module so above is in fact
an isomorphism, and not only almost-isomorphism.
Taking the limit degree of both sides over $k(y)$ gives
the equality
$$
\dl(X/Y)=e_1\dl(\kk(x_1)/\kk(y))+\cdots+e_r\dl(\kk(x_r)/\kk(y)),
$$
as required. 
\end{proof}

The above result is not completely satisfactory as it carries the assumption 
that $\varphi^{-1}(y)\neq\emptyset$, which typically will  not be satisfied quite often,
for a dense set of $y$'s. We need to find a situation in which the morphism
$\varphi$ can be made surjective, and the solution is offered by the framework
of \emph{generalised difference schemes}.

\begin{theorem}
Let $\varphi:(X,\Sigma)\to (Y,\sigma)$ be a strongly $\Sigma$-finite morphism
of non-singular curves
which is a generic Galois covering. Then, for every $y\in Y$,
$$
\deg(\varphi^*(y))=|\Sigma|\dl(X/Y).
$$
\end{theorem}

\begin{proof}
Using standard reductions as in the previous proofs, we reduce the
problem to the following algebraic situation.
We have a morphism $(B,\sigma)\to(A,\Sigma)$ making $(A,\Sigma)$
$\Sigma$-finite over $(B,\sigma)$, and their associated difference curves
are non-singular, so we may assume that $A$ and $B$ are normal domains
with fraction fields $L$ and $K$. We also know that $\Sigma$ is a finite set
of representatives of isomorphism classes of all lifts of $\sigma$ from $K$ to $L$
and that $L$ is Galois over $K$.
Using Babbitt's Decomposition for the extension $(K,\sigma)\to(L,\Sigma)$, we obtain
a sequence of difference field extensions
$$
(K,\sigma)\to(L_0,\Sigma)\to(L_1,\Sigma)\to\cdots\to(L_n,\Sigma)=(L,\Sigma),
$$
such that $L_0$ is finite over $K$ and each $L_{i+1}$ is benign over $L_i$ for 
$i\geq 0$. Let $A_i$ be the integral closure of $B$ in $L_i$. By assumptions
of $\sigma$-finiteness, we know that $A_n=A$.
By Theorem 13.14 in \cite{eisenbud-comm-alg},  $A_0$ is a $B$-lattice in $L_0$.
Thus, if we consider $y\in\spec^\sigma(B)=Y$, 
$(B\setminus I_y)^{-1}A_0$ is integral over $\OO_y$ and is contained in a finite $\OO_y$-module. By non-singularity, $\OO_y$ is noetherian so $(B\setminus I_y)^{-1}A_0$
is a finite $\OO_y$-module and we conclude that $(\varphi^{-1}(y),\Sigma)\to (y,\sigma)$
is a Galois covering and thus onto.
Moreover, from the assumptions, for every $\tilde{\sigma}\in\Sigma$, 
$(A_n,\tilde{\sigma})\to(A_0,\tilde{\sigma})$ is $\sigma$-finite and surjective,
as a tower of benign extensions. 
We can finish the proof in two different ways.

The first is to note that the statement is compatible with taking composites
so it suffices to check it for the morphism $(A_n,\Sigma)\to (A_0,\Sigma)$
and for $(A_0,\Sigma)\to(B,\sigma)$, where the former follows by applying
\ref{degdl}, and the latter follows along the lines of the usual proof of the
corresponding statement for algebraic curves.

The second way is to apply \ref{degdl} to each $(A_n,\tilde{\sigma})\to(B,\sigma)$
for $\tilde{\sigma}\in\Sigma$ while making sure to account for the fact that the ramification index at
$x\in \spec^\Sigma(A_0)$ in the morphism associated with 
$(B,\sigma)\to (A_0,\Sigma)$  equals the size of 
$\Sigma_x=\{\tilde{\sigma}\in\Sigma:\tilde{\sigma}(x)=x\}$.

\end{proof}

\bibliographystyle{plain}

\end{document}